\documentclass[reqno,centertags,12pt,draft]{amsart}
\usepackage[letterpaper,margin=1.3in]{geometry}
\usepackage{amsmath,amsthm,amsfonts,amssymb,enumerate}
\usepackage[bookmarksopen=true,final]{hyperref}
\usepackage{multirow,yhmath}
\usepackage{xcolor}

\newcommand{\bb}{\mathbb}

\renewcommand{\Re}{\mathrm{Re}}
\renewcommand{\Im}{\mathrm{Im}}

\newtheorem{theorem}{Theorem}

\newtheorem{proposition}[theorem]{Proposition}

\theoremstyle{definition}

\theoremstyle{remark}

\numberwithin{equation}{section}
\numberwithin{theorem}{section}
\theoremstyle{remark}
\newtheorem*{remarks}{Remarks}

\newcommand{\bbR}{\mathbb{R}}
\newcommand{\bbC}{\mathbb{C}}
\newcommand{\calJ}{\mathcal{J}}

\begin{document}
	\title[Perturbations of chiral Gaussian $\beta$-ensembles]{Hermitian and non-Hermitian perturbations of chiral Gaussian $\beta$-ensembles}

\author{G\"{o}kalp Alpan}
\address{Department of Mathematics, Uppsala University, Uppsala, Sweden}
\email{gokalp.alpan@math.uu.se}
	
\author{Rostyslav Kozhan}
\address{Department of Mathematics, Uppsala University, Uppsala, Sweden}
\email{rostyslav.kozhan@math.uu.se}
\keywords{Random matrices, eigenvalues, chiral ensembles, Jacobi matrices}
	\begin{abstract}
We compute the joint eigenvalue distribution for the rank one Hermitian and non-Hermitian perturbations of chiral Gaussian $\beta$-ensembles  ($\beta>0$) of random matrices. 

	\end{abstract}
	
	\date{\today}
	\maketitle
	
	\section{Introduction}\label{ss:Intro}
	
	

Let $X$ be an $m\times n$ matrix with entries being i.i.d. real ($\beta=1$), complex ($\beta=2$), or quaternionic ($\beta=4$) centered normal random variables with  $\mathbb{E}(|X_{11}|^2) = \beta$. 
Then we say that the $(m+n)\times(m+n)$ Hermitian matrix 
\begin{equation}\label{eq:chiral}
	 {H} = \begin{pmatrix}
	\textbf{0}_{m\times m} & X\\\
	X^*& \textbf{0}_{n\times n}
\end{pmatrix}.
\end{equation}
belongs to the chiral Gaussian  orthogonal ($\beta=1$),  unitary $(\beta=2$),  symplectic ($\beta=4$) random matrix ensemble (chGOE, chGUE, 	chGSE, respectively).

In this paper we find explicitly the joint eigenvalue distribution of rank one Hermitian and non-Hermitian perturbations of chiral ensembles:

\begin{equation}\label{eq:chiralPert}
	\widetilde{H}= \begin{pmatrix}
	\Gamma & X\\\
	X^* &\textbf{0}_{n\times n}
\end{pmatrix}.
\end{equation}
Here $\Gamma$ is an $m\times m$ matrix with $\operatorname{rank}\,\Gamma =1$ and either $\Gamma=\Gamma^*$ (Hermitian perturbation) or $\Gamma=-\Gamma^*$ (anti-Hermitian perturbation). 
The matrix $\Gamma$ can be either deterministic or random but independent from $X$. We will also allow arbitrary $\beta>0$ different from $\beta=1,2,4$ (see Section~\ref{ss:Models} for details). 
 
 The main results are Theorems~\ref{th:Hermitian} and~\ref{th:nonHermitian} for Hermitian and non-Hermitian perturbations, respectively. We use methods developed in~\cite{KK,Koz17,Koz20}. Namely, first, we develop sparse (Jacobi) matrix models for chiral ensembles and their perturbations in the spirit of Dumitriu--Edelman~\cite{Dumede02} (see Section~\ref{ss:Models}). This allows us to use the theory of orthogonal polynomials and Jacobi matrices to compute a Jacobian of a certain change of variables (Section~\ref{ss:Jacobians}) which leads to the desired distribution (Sections~\ref{ss:Hermitian} and~\ref{ss:nonHermitian}).
 
By multiplying matrices in~\eqref{eq:chiral} and~\eqref{eq:chiralPert} by  $i$, and letting $Y=iX$, $\Lambda=i\Gamma$, we can equivalently work with the chiral Gaussian anti-Hermitian model
$$\begin{pmatrix}
	\textbf{0}_{m\times m} & Y\\\
	-Y^* &\textbf{0}_{n\times n}
\end{pmatrix}
$$
 and its Hermitian $\Lambda=\Lambda^*$ and anti-Hermitian $\Lambda=-\Lambda^*$ perturbations
\begin{equation}\label{eq:antiHermPert}
 \begin{pmatrix}
 	\Lambda & Y\\\
 	-Y^* &\textbf{0}_{n\times n}
 \end{pmatrix}.
\end{equation}
All the results in this paper can be trivially restated for this case: all the matrix models and eigenvalues simply get a factor of $i$. The benefit of this would be that the characteristic polynomial of~\eqref{eq:antiHermPert} in the case $\Lambda=\Lambda^*$  has real coefficients (instead of alternating between purely imaginary and purely real as in Section~\ref{ss:nonHermitian}), so its zeros belong to $\{z:\Re z<0\}$ and are symmetric with respect to $\bbR$.

Chiral random matrix theory has been an important instrument in quantum chromodynamics (QCD), going back to works \cite{va1,va0,va3}, see ~\cite{ake,Dam,va4,va2} for overviews, lecture notes, and further references. 

There is a vast literature on low rank {\it non-Hermitian} perturbations of Hermitian random matrices, owing to its physical applications in quantum chaotic scattering. For an overview, physical applications, and references, we refer readers to the papers \cite{fyo16,fyosav15,fyosom03,nucl}. The exact eigenvalue distribution of low rank non-Hermitian  perturbations of Gaussian and Laguerre $\beta$-ensembles was the topic of \cite{fyokho99,Koz17,Koz20,SokZel,StoSeb,Ull} in particular. 

The low rank non-Hermitian perturbations of chiral ensembles that we study here do not seem to have been studied in the literature before. A different type of non-Hermitian perturbations (of full rank) have been studied recently in~\cite{kie}.

Literature that studies {\it Hermitian} perturbations of Gaussian and Laguerre random matrix ensembles is huge. 

The additive model $H+\Gamma$ for perturbations of Gaussian random matrices $H$ bears the name Gaussian with an external source or shifted mean Gaussian ensemble, see~\cite{BleKuj,BreHik96,Pastur72,Zinn1,Zinn2} among many others. 

The usual model for perturbations of Laguerre ensembles is $ (I+\Gamma)^{1/2} X^* X  (I+\Gamma)^{1/2}$ with $\Gamma=\Gamma^*$ of low rank. This is typically referred to as the spiked Wishart  ensembles, see, e.g., \cite{BBP,BGN1,DesFor06,Johnstone}.  
Clearly this corresponds to 
perturbation  $X\mapsto X(I+\Gamma)^{1/2}$ and $X^*\mapsto (I+\Gamma)^{1/2} X^*$  
in the chiral model~\eqref{eq:chiral}. 

 Another type of perturbation of Laguerre/Wishart ensembles  actively studied in the literature is $(X+\Gamma)^* (X+\Gamma)$. This corresponds to $X\mapsto X+\Gamma$ and $X^*\mapsto (X+\Gamma)^*$  
 in~\eqref{eq:chiral} which bear the name chiral Gaussian ensembles with a source, see e.g.~\cite{DesFor08,For13,FGS18,SWG9} and~\cite[Sect 11.2.2]{Forrester-book}. 

We stress that eigenvalues of our Hermitian perturbed model~\eqref{eq:chiralPert}, however, do {\it not} correspond to a change of variables applied to eigenvalues of a simple perturbation of the Laguerre random matrix.



{ \bf Acknowledgments}: Research of G. A. was supported by Vergstiftelsen foundation.

\section{Jacobi matrix models}\label{ss:Models}
\subsection{Jacobification: case $m\le n$}\label{ss:chiralModel1}
As was shown by Dumitriu--Edelman~\cite{Dumede02}, $X$ can be bidiagonalized in the following sense: there are $m\times m$ and $n \times n$ unitary matrices $L$ and $R$ such that
\begin{equation}\label{BX}
	B := L X R=\left( \begin{matrix}
		x_1 &  & & & & \multicolumn{3}{|c}{}  \\
		y_1 & x_2 & & & & \multicolumn{3}{|c}{} \\
		& y_2 & \ddots  & & & \multicolumn{3}{|c}{\mathbf{0}_{m\times (n-m)}}\\
		& & \ddots &\ddots & & \multicolumn{3}{|c}{}  \\
		& & & y_{m-1} & x_m & \multicolumn{3}{|c}{}
	\end{matrix}\right)
\end{equation}
 with
 \begin{equation}\label{eq:L}
 	L e_1= L^*e_1=e_1,
 \end{equation}
 where $e_1$ is the vector with $1$ in the first entry and $0$ in all others. Here the $x_j$'s and $y_j$'s are independent random variables with the distributions
  \begin{align}
 \label{eq:x's}
 	x_j 
 	&\sim \chi_{\beta(n-j+1)}, \,\,\,\, 1\leq j\leq m,\\
 	 \label{eq:y's}
 	y_j &\sim \chi_{\beta(m-j)}, \,\,\,\, 1\leq j\leq m-1,
 \end{align}
where $\chi_\alpha$ stands for the chi-distributed random variable with parameter $\alpha>0$ given by the p.d.f. 
$\tfrac{1}{2^{\alpha/2-1}\Gamma(\tfrac{\alpha}{2})} x^{\alpha-1} e^{-x^2/2}$ for $x > 0$.
 
 Trivially, ~\eqref{BX} implies $B^* := R^* X^* L^* $ which means that our chiral matrix $H$ from~\eqref{eq:chiral} can be unitarily  reduced to
\begin{equation}\label{eq:chiralModelInterm}
\begin{pmatrix}
	L & \textbf{0}_{m\times n}\\\
	\textbf{0}_{n\times m} & R^*
\end{pmatrix}  H \begin{pmatrix}
L^* & \textbf{0}_{m\times n}\\\
\textbf{0}_{n\times m} & R
\end{pmatrix}
= 
\begin{pmatrix}
	\textbf{0}_{m\times m} &B\\\
	B^*& \textbf{0}_{n\times n}
\end{pmatrix}
\end{equation}
The right-hand side is a sparse matrix with independent entries. However we want a Jacobi (tridiagonal) form in order to employ theory of orthogonal polynomials. To this end, we introduce  the  $(m+n)\times (m+n)$  permutation matrix $P$ corresponding to the permutation
\begin{equation}\label{eq:permut}
	\begin{pmatrix}
		1 & 2   &  3 & 4   & \cdots & 2m-1 & 2m & \multicolumn{1}{|c}{2m+1} & \cdots & m+n \\
		1 & m+1 & 2  & m+2 & \cdots & m    & 2m & \multicolumn{1}{|c}{2m+1} & \cdots & m+n 
	\end{pmatrix}.
\end{equation}
This produces
\begin {equation}\label{ddd}
P\begin{pmatrix}
	\textbf{0}_{m\times m} &B\\\
	B^*&  \textbf{0}_{n\times n}
\end{pmatrix} P^*=
\left(
\begin{matrix}
	0 & x_1 &&&&& \multicolumn{3}{|c}{}\\
	x_1 & 0 & y_1 &&&& \multicolumn{3}{|c}{} \\
	& y_1 & 0 & x_2 &&& \multicolumn{3}{|c}{ \multirow{2}{*}{$\mathbf{0}_{2m\times(n-m)} $} }  \\
	& & \ddots & \ddots & \ddots && \multicolumn{3}{|c}{} \\
	&& & y_{m-1} & 0 & x_m & \multicolumn{3}{|c}{}\\
	&& & & x_m &0 & \multicolumn{3}{|c}{} \\
	\hline
	\multirow{2}{*}{} &  \multicolumn{4}{c}{\multirow{2}{*}{$\mathbf{0}_{(n-m)\times 2m}$}}  &\multirow{2}{*}{} & 
	\multicolumn{3}{|c}{\multirow{2}{*}{$\mathbf{0}_{(n-m)\times (n-m)}$}}
	\\ 
	&&&&&& \multicolumn{3}{|c}{}
	\\
	\end{matrix}
\right) = : J.
\end{equation}
Observe also that 
\begin{equation}\label{eq:P}
P e_1=P^*e_1=e_1, \quad P  I_{1\times 1} P^*=I_{1\times 1},
\end{equation}
where $I_{1\times 1}$ is the diagonal matrix with 1 in $(1,1)$-entry and $0$ everywhere else. We will use these properties later in the text.

%

 This ensemble already appeared earlier in \cite{jac}, see also \cite{dum2}.
 
 \subsection{Jacobification: case $m\ge n+1$}\label{ss:chiralModel2}
 Arguments of Dumitriu--Edelman work for the case $m\ge n+1$ with the following modifications:~\eqref{BX} becomes
 \begin{equation}\label{BX2}
 	B := L X R=\left( \begin{matrix}
 		x_1 &  & & &   \\
 		y_1 & x_2 & & &  \\
 		& y_2 & \ddots  & & \\
 		& & \ddots &\ddots &  \\
 		& & & y_{n-1} & x_n  \\
 		& & & & y_n \\
 		\hline
 		\multirow{2}{*}{} & \multirow{2}{*}{} &\multirow{2}{*}{$\textbf{0}_{m-n-1,n}$} &\multirow{2}{*}{} & \multirow{2}{*}{} \\ \\
  		 	\end{matrix}\right);
 \end{equation}
 distributions of $x_j$'s and $y_j$'s are now
  \begin{align}
 	\label{eq:x's2}
 	x_j 
 	&\sim \chi_{\beta(n-j+1)}, \,\,\,\, 1\leq j\leq n,\\
 	\label{eq:y's2}
 	y_j 
 	&\sim \chi_{\beta(m-j)}, \,\,\,\, 1\leq j\leq n;
 \end{align}
 equation~\eqref{eq:chiralModelInterm} remains unchanged; the permutation matrix $P$ in~\eqref{eq:permut} is now
 \begin{equation}\label{eq:permut2}
 	\begin{pmatrix}
 		1 & 2   &  3 & 4   & \cdots & 2n-1 & 2n & \multicolumn{1}{|c}{2n+1} & \cdots & m+n \\
 		1 & m+1 & 2  & m+2 & \cdots & n    & n+m & \multicolumn{1}{|c}{n+1} & \cdots & m 
 	\end{pmatrix};
 \end{equation}
finally,~\eqref{ddd} becomes
\begin {equation}\label{ddd2}
P\begin{pmatrix}
	\textbf{0}_{m\times m} &B\\\
	B^*&  \textbf{0}_{n\times n}
\end{pmatrix} P^*=
\left(
\begin{matrix}
	0 & x_1 &&&&& \multicolumn{3}{|c}{}\\
	x_1 & 0 & y_1 &&&& \multicolumn{3}{|c}{} \\
	& y_1 & 0 & x_2 &&& \multicolumn{3}{|c}{ \multirow{2}{*}{$\mathbf{0}_{(2n+1)\times(m-n-1)} $} }  \\
	& & \ddots & \ddots & \ddots && \multicolumn{3}{|c}{} \\
	&& & x_{n} & 0 & y_n & \multicolumn{3}{|c}{}\\
	&& & & y_n &0 & \multicolumn{3}{|c}{} \\
	\hline
	\multirow{2}{*}{} &  \multicolumn{4}{c}{\multirow{2}{*}{$\mathbf{0}_{(m-n-1)\times (2n+1)}$}}  &\multirow{2}{*}{} & 
	\multicolumn{3}{|c}{\multirow{2}{*}{$\mathbf{0}_{(m-n-1)\times (m-n-1)}$}}
	\\ 
	&&&&&& \multicolumn{3}{|c}{}
	\\
\end{matrix}
\right) = : J.
\end{equation}


\subsection{Chiral Gaussian $\beta$-ensembles}\label{ss:chiral}

In the previous two subsections we have obtained that $H$ from~\eqref{eq:chiral} is unitarily equivalent to the Jacobi matrix $J$ in~\eqref{ddd}  with~\eqref{eq:x's}--\eqref{eq:y's}.

It will occasionally be convenient to have a notation for the same Jacobi matrix but without the last zero block. So let us introduce the matrix $\calJ$ which is obtained by removing the last $n-m$ of zero rows and columns of $J$ in~\eqref{ddd} ($m\le n$) or the last $m-n+1$ of zero rows and columns   in~\eqref{ddd2}  ($m\ge n+1$). We obtain the $N\times N$ Jacobi matrix
\begin{equation}\label{generalJ}
	\calJ:= \begin{pmatrix}
		0 & a_1 \\
		a_1 & 0 & a_2 \\
		& a_2 & 0 & \ddots \\
		& & \ddots & \ddots & a_{N-1}\\
		&& & a_{N-1} & 0 
	\end{pmatrix},
\end{equation}
where $a_{2j-1} = x_j$, $a_{2j}=y_j$ and either $N=2m$,~\eqref{eq:x's}--\eqref{eq:y's} ($m\le n$) or $N=2n+1$,~\eqref{eq:x's2}--\eqref{eq:y's2} ($m\ge n+1$).

We will say that  $\calJ$ belongs to the \textbf{chiral Gaussian $\beta$-ensemble}, chG$\beta$E for short. This ensemble makes sense for arbitrary $\beta>0$, not just $\beta=1,2,4$. 

 \subsection{Rank one Hermitian perturbations}
 Now we consider the  perturbed model \eqref{eq:chiralPert} with Hermitian $\Gamma$. Since $\Gamma$ has rank 1, we can choose $\Gamma$ to be positive semi-definite. Let 
 \begin{equation}\label{l}
 	l= \|\Gamma\|_{{HS}}:= \left(\sum_{j,k=1}^m |\Gamma_{jk}|^2\right)^{1/2}.
 \end{equation}
 be the Hilbert--Schmidt norm of the perturbation. 
 
 \begin{proposition}\label{pr:Model1}
 	Let $\widetilde{H}$ be as in~\eqref{eq:chiralPert}. Assume that $\Gamma=\Gamma^*\ge {\normalfont \textbf{0}}_{m\times m}$ has $\operatorname{rank}\,\Gamma =1$ and $||\Gamma||_{HS} = l$. Further assume that $\Gamma$ has real, complex, quaternionic entries for $\beta=1,2,4$, respectively, that are either deterministic or random but independent from $X$.  Then  $\widetilde{H}$ is unitarily equivalent to
 	\begin{equation}\label{eq:Model1}
 	J+ l I_{1\times 1},
	\end{equation}
	where $J$ is~\eqref{ddd} or ~\eqref{ddd2}. 
 \end{proposition}
\begin{remarks}
	1. We will consider~\eqref{eq:Model1} for general $\beta>0$ and view it as the rank one Hermitian perturbation of the chiral Gaussian $\beta$-ensemble from Subsection~\ref{ss:chiral}. In fact, we will remove the zero block and will be working with $\calJ+ l I_{1\times 1}$.
	
	
	2. The trick in the proof with reducing  rank one perturbation to $(1,1)$-entry which carries through to the Jacobi matrix model is well known: it has been used in~\cite{KK,Koz17,Koz20}, and even earlier by Bloemendal--Vir\'{a}g~\cite{BloVir} in their study of spiked Laguerre ensembles.
\end{remarks}

\begin{proof} 
   $\Gamma$ can be represented as $\Gamma= U (lI_{1\times 1}) U^*$ for some $m\times m$ matrix $U$ which is orthogonal, unitary, or unitary symplectic for $\beta=1,2,4$, respectively. 
 
 Then the matrix $\widetilde{H}$ (see~\eqref{eq:chiralPert}) satisfies
 \begin {equation}\label{equi1}
 \begin{pmatrix}
 	U^* & \textbf{0}_{m\times n}\\\
 	\textbf{0}_{n\times m}& \textbf{1}_{n\times n}
 \end{pmatrix}
 \widetilde{H}
 \begin{pmatrix}
 	U & \textbf{0}_{m\times n}\\\
 	\textbf{0}_{n\times m} & \textbf{1}_{n\times n}
 \end{pmatrix}= 
 \begin{pmatrix}
 	\textbf{0}_{m\times m} & U^*X\\\
 	(U^*X)^* & \textbf{0}_{n\times n}
 \end{pmatrix}+lI_{1\times 1}.
\end{equation}
Here $\textbf{1}_{n\times n}$ stands for the $n\times n$ identity matrix. Now, note that $U$ is independent of $X$, so the joint distribution of the elements of $Y=U^*X$ is identical to the distribution of $X$ by Gaussianity. 
Hence we can apply the arguments from Subsection~\ref{ss:chiralModel1}/\ref{ss:chiralModel2} but to $Y$ instead of $X$ to arrive at
\begin {multline}\label{equi2}
P \begin{pmatrix}
	L & \textbf{0}_{m\times n}\\\
	\textbf{0}_{n\times m} & R^*
\end{pmatrix}
\left(
\begin{pmatrix}
	\textbf{0}_{m\times m} &Y\\\
	Y^*& \textbf{0}_{n\times n}
\end{pmatrix}
+lI_{1\times 1}\right)
\begin{pmatrix}
	L^* & \textbf{0}_{m\times n}\\\
	\textbf{0}_{n\times m} & R
\end{pmatrix}
P^*
\\
= 
P\begin{pmatrix}
	\textbf{0}_{m\times m} &B\\\
	B^*& \textbf{0}_{n\times n}
\end{pmatrix}
P^*
+l PI_{1\times 1}P^*
=
J+ l I_{1\times 1},
\end{multline}
where we have used~\eqref{eq:L} and~\eqref{eq:P}.
\end{proof}

 \subsection{Rank one non-Hermitian perturbations}
 In the exact same way, we can consider the perturbed model~\eqref{eq:chiralPert} with anti-Hermitian $\Gamma$. 
 
  \begin{proposition}\label{pr:Model2}
  	Let $\widetilde{H}$ be as in~\eqref{eq:chiralPert}.  Assume $\Gamma=-\Gamma^*$, $(-i\Gamma)\ge {\normalfont \textbf{0}}_{m\times m}$, $\operatorname{rank}\,\Gamma =1$ and $||\Gamma||_{HS} = l$. Further assume that $i\Gamma$ has real, complex, quaternionic entries for $\beta=1,2,4$, respectively, that are either deterministic or random but independent from $X$. Then  $\widetilde{H}$ is unitarily equivalent to
 	\begin{equation}\label{eq:Model2}
 	J+ i l I_{1\times 1},
 	\end{equation}
 	where $J$ is~\eqref{ddd} or ~\eqref{ddd2}.
 \end{proposition}
 
 \begin{proof} 
 	Notice that $-i\Gamma$ is Hermitian positive semi-definite and of rank one, so $-i\Gamma= U (lI_{1\times 1}) U^*$ for some $m\times m$ matrix $U$ which is orthogonal, unitary, or unitary symplectic for $\beta=1,2,4$, respectively. The rest of the arguments go through without any changes.
 \end{proof}

\subsection{Anti-bidiagonal models}
Matrix model $\calJ+l I_{1\times 1}$ (as well as $\calJ+il I_{1\times 1}$, of course) can be also represented in the so-called anti-bidiagonal form. To do so, we introduce another permutation matrix

$$
Q=
\begin{pmatrix}
	1 & 2 & 3 & \cdots & N-2 & N-1 & N \\
	N & N-2 & N-4& \cdots &  N-5  & N- 3& N-1
\end{pmatrix}.
$$


Then
\begin{equation}\label{eq:antibidiagonal}
Q \left(\calJ+l I_{1\times 1}\right) Q^*
=
\begin{pmatrix}
	0 & 0 & \cdots  & 0 & a_{N-1} \\
	0 & 0 & \cdots & a_{N-3} & a_{N-2} \\
	\vdots & \vdots &  & \vdots & \vdots \\
	0 	&a_{N-3} & \cdots & 0 & 0\\
	a_{N-1} & a_{N-2}& \cdots& 0 & 0 
\end{pmatrix}
\end{equation}
This matrix has two anti-diagonals with the perturbation term $l$ being now ``in the middle'' at the position $(\lfloor \tfrac{N}{2} \rfloor +1,\lfloor \tfrac{N}{2} \rfloor +1)$.

\section{Location of the eigenvalues}
In the next two statements, we find all the possible configurations of eigenvalues for our perturbed Jacobi ensembles~\eqref{eq:Model1}, ~\eqref{eq:Model2}. Even more is true: every possible configuration of eigenvalues occurs exactly once. 

\begin{proposition}\label{bijection2}
Let $N>1$. Then there is a one-to-one correspondence between $N$ points $z_1, z_2,\ldots, z_{N}$ with $z_1>-z_2>z_3>\cdots>(-1)^{N-1} z_{N}$ and the matrices $\calJ+lI_{1\times 1}$ where $\calJ$ is of the form \eqref{generalJ} and $a_1,\ldots, a_{N-1},l >0$.
\end{proposition}
\begin{proof}
	This was shown by Holtz \cite[Corollary 2]{Holtz} who classified eigenvalues of matrices~\eqref{eq:antibidiagonal}. 
\end{proof}

\begin{proposition}\label{bijection}
Let $N>1$. Then there is a one-to-one correspondence between $N$ points $z_1, z_2,\ldots, z_{N}$ in $\bb C_+:=\{z\in \bb C: \Im z>0   \}$ (counting multiplicity)  that are symmetric with respect to the imaginary axis and the matrices  $\calJ+ilI_{1\times 1}$ where $J$ is of the form \eqref{generalJ}
and $a_1,\ldots, a_{N-1},l >0$.
\end{proposition}
\begin{proof}

Let $z_1,\ldots, z_{N}$ be $N$ points in $\bb C_+$ that are symmetric with respect to the imaginary axis. By the results of Arlinski\u{\i}--Tsekanovski\u{\i}~\cite[Theorem~5.1, Corollary~6.5]{ArlTse06}, there is a $\calJ$ of the form \eqref{generalJ} and $l>0$ such that $z_1,\ldots, z_{N}$ are the eigenvalues of  $\calJ+ilI_{1\times 1}$.

Conversely, let $\calJ$ be a Jacobi matrix as in \eqref{generalJ} and $l>0$. By~\cite[Prop~4.1]{ArlTse06} eigenvalues of  $\calJ+ilI_{1\times 1}$ belong to $\bb C_+$. Since $-(\calJ+ilI_{1\times 1})^*=W(\calJ+ilI_{1\times 1})W^*$, where $W$ is the diagonal unitary matrix with diagonal $\{1,-1,1-1,\ldots\}$, we obtain the symmetry of the eigenvalues with respect to the imaginary axis. 



\end{proof}


\section{Spectral measures of chiral Gaussian $\beta$-ensembles}
Given an $k\times k$ Hermitian matrix $H$, define its spectral measure with respect to $e_1$ to be the probability measure $\mu$ satisfying
\begin{equation}\label{eq:spmeas}
\langle e_1, H^j  e_1 \rangle = \int_\bbR x^j d\mu(x), \quad \mbox{ for all } j\in  \bb Z_{\ge 0}.
\end{equation}
We will refer to it as simply ``the spectral measure'' from now on. By diagonalizing $H$ and assuming $e_1$ is cyclic, one can see that 
$$
\mu = \sum_{j=1}^k w_j \delta_{\lambda_j}
$$
with $\sum_{j=1}^k w_j = 1$ and $w_j>0$. Here $\{\lambda_j\}_{j=1}^k$ are the eigenvalues of $H$ (which are distinct by cyclicity), and $w_j = |\langle v_j,e_1\rangle|^2$, where $v_j$ is the corresponding eigenvector.

Now let us assume that $H$ is from the chGOE, chGUE, or chGSE. 
As we show in Subsections~\ref{ss:chiralModel1} and~\ref{ss:chiralModel2}, $H$ and $J$ are unitarily equivalent $H=UJU^*$. Moreover, $Ue_1=U^*e_1=e_1$ implies that they have identical spectral measures. Finally, spectral measures of $J$ and $\calJ$ coincide, which can be trivially seen from~\eqref{eq:spmeas}. In the next theorem, we compute this common spectral measure. The result works for any $\beta>0$.
\begin{theorem}\label{thmmm}
	For $\beta>0$ let $\calJ$ belong to chG$\beta$E 
	$($see Subsection~\ref{ss:chiral}$)$. Let $a=|n-m|+1-2/\beta$.
	\begin{enumerate}[$(i)$]
		\item If $m\le n$ $($that is, $\calJ$ is $2m\times 2m$$)$, then with probability $1$ the spectral measure of $\calJ$ is:
		\begin{equation}\label{eq:spMeasDiscrete}
		\mu=\sum_{j=1}^m \tfrac12 w_j (\delta_{\lambda_j} + \delta_{-\lambda_j})
		\end{equation}
		with the joint distribution of $\lambda_1,\ldots, \lambda_m$, $w_1,\ldots ,w_{m-1}$  given by
		\begin{align}
			\frac{2^m}{h_{\beta,m,a}}&\prod_{j=1}^m\lambda_j^{\beta a+1}e^{-\lambda_j^2/2}\prod_{1\leq j<k\leq m}|\lambda_k^2-\lambda_j^2|^\beta d\lambda_1\dots d\lambda_m              \\
			&\times  \frac{\Gamma(\beta m/2)}{\Gamma(\beta/2)^m}\prod_{j=1}^m w_j^{\beta/2-1} d w_1\dots d w_{m-1}.
		\end{align}
				
		\item If $m \ge n+1$ $($that is, $\calJ$ is $(2n+1)\times(2n+1)$$)$, then with probability $1$ the spectral measure of $\calJ$ is:
		\begin{equation}\label{eq:spMeasDiscrete2}
		\mu=w_0\delta_{0}+\sum_{j=1}^n \tfrac12 w_j (\delta_{\lambda_j} + \delta_{-\lambda_j})
		\end{equation}
		with the joint distribution of $\lambda_1,\ldots, \lambda_n$, $w_1,\ldots ,w_n$ given by
		\begin{align}
			\frac{2^n}{h_{\beta,n,a}}&\prod_{j=1}^n\lambda_j^{\beta a+1}e^{-\lambda_j^2/2}\prod_{1\leq j<k\leq n}|\lambda_k^2-\lambda_j^2|^\beta d\lambda_1\dots d\lambda_n             \\
			&\times  \frac{\Gamma(\beta m/2)}{\Gamma(\beta/2)^n \Gamma{(\beta(m-n)/2)}} 
			w_0^{\beta(m-n)/2-1}
			\prod_{j=1}^n w_j^{\beta/2-1}d w_1\dots d w_{n}.
		\end{align}
	\end{enumerate}
 Here the normalization constant is
	\begin{equation}
	h_{\beta,s,a}=2^{s(a\beta/2+1+(s-1)\beta/2)}\prod_{j=1}^s\frac{\Gamma(1+\beta j/2)\Gamma{(1+\beta a/2+\beta(j-1)/2)}}{\Gamma{(1+\beta/2)}}.
\end{equation}
\end{theorem}
\begin{proof}


Jacobi matrices~\eqref{generalJ} with non-zero $a_j$'s have simple spectrum. From this and symmetry, we then get that for  $m\leq n$, $\calJ$ has $m$  distinct positive eigenvalues $\lambda_1,\ldots, \lambda_m$ and $m$ distinct negative eigenvalues $-\lambda_1,\ldots, -\lambda_m$, so the spectral measure of $\calJ$ has form~\eqref{eq:spMeasDiscrete}.


Similarly, if $m\geq n+1$ then $\calJ$ has $n$ distinct positive eigenvalues  $\lambda_1,\ldots, \lambda_n$, $n$ distinct negative eigenvalues $-\lambda_1,\ldots, -\lambda_n$ and a simple eigenvalue at $\lambda_0:=0$. Consequently, the spectral measure of $\calJ$ has form~\eqref{eq:spMeasDiscrete2}.

Notice that the matrix 
\begin{equation}\label{eq:chiralWithB}
G=\begin{pmatrix}
	\textbf{0}_{m\times m} &B\\\
	B^*&  \textbf{0}_{n\times n}
\end{pmatrix}
\end{equation}
is unitarily equivalent to $J$: see~\eqref{ddd} and~\eqref{ddd2}. Moreover, because of ~\eqref{eq:P}, $G$ has the same spectral measure as $J$, $\calJ$.

For $k\neq 0$, we can write a normalized eigenvector of $G$  corresponding to $\lambda_k$ in the form 
\begin{equation}
	\begin{pmatrix}
		u^{(k)}\\
		v^{(k)}
	\end{pmatrix}
\end{equation}
so that 

\begin{align}
	BB^*u^{(k)}&=\lambda_k^2 u^{(k)},\\
	B^*Bv^{(k)}&=\lambda_k^2 v^{(k)}
\end{align}
are satisfied. Note that 
\begin{equation}
	\begin{pmatrix}
		u^{(k)}\\
		-v^{(k)}
	\end{pmatrix}
\end{equation}
is a normalized eigenvector of $G$ associated with $-\lambda_k$. By orthononormality of  the eigenvectors we have 

\begin{align}
&	\|u^{(k)}\|^2+\|v^{(k)}\|^2=1,\\
&	\|u^{(k)}\|^2-\|v^{(k)}\|^2=0,
\end{align}
and thus $\|u^{(k)}\|^2=1/2$. Recall~\eqref{eq:spMeasDiscrete},~\eqref{eq:spMeasDiscrete2} that denoted the eigenweight on $\lambda_k$ by $w_k/2$  (for $k\neq 0$). Then $w_k=2 |\langle u^{(k)},e_1\rangle|^2$.

For $k>0$, let
\begin{equation}\label{eig1}
	\lambda_k^\prime:=\lambda_k^2
\end{equation}
and $w_k^\prime$ be the eigenweight for $BB^*$ at $\lambda_k^\prime$. Then $\sqrt{2}u^{(k)}$ is a normalized eigenvector for  $BB^*$ corresponding to $\lambda_k^\prime$. Thus,
\begin{equation}\label{eig2}
	w_k^\prime=2 |\langle u^{(k)},e_1\rangle|^2=w_k.
\end{equation}
For the case (ii) we also have $w'_0=1-\sum_{k=1}^n w'_k=1-\sum_{k=1}^n w_k= w_0$.

Finally, recall that the joint distribution of $\{\lambda_k^\prime\}$ and $\{w_k^\prime\}$ of  $BB^*$  and of the $\beta$-Laguerre random matrix coincide (\cite{Dumede02}, \cite[Lemma 4]{Koz17}, \cite[Proposition 1]{Koz17}). 
Using \eqref{eig1}, \eqref{eig2} we can therefore write the joint distribution of the $\lambda_k$'s and $w_k$'s.
\end{proof}


\section{Jacobians}\label{ss:Jacobians}

We fix $l>0$ and for $\calJ$ as in~\eqref{generalJ} let 
\begin{align}\label{eq:Jl}
	\calJ_l & = \calJ+lI_{1\times 1}, \\
	\label{eq:Jil}
	\calJ_{il} & = \calJ+ilI_{1\times 1}.
\end{align}
In this section, we compute the Jacobian(s) of the change of variables from the spectral parameters (that is, $\lambda_j$'s and $w_j$'s) to the Maclaurin coefficients $\kappa_j$'s of the characteristic polynomial $\kappa(z)$ of $\calJ_l$ or $\calJ_{il}$.

\begin{theorem}\label{eigd2}
	
	Let $l>0$.
	\begin{enumerate}[$(i)$]
		\item Let $\calJ$ be a $2m\times 2m$ Jacobi matrix of the form \eqref{generalJ} with $a_1,\ldots,a_{2m-1}>0$ and $m>0$. Denote $\mu$ to be its spectral measure~\eqref{eq:spMeasDiscrete}.
		Let 
		\begin{equation}\label{det11}
			\det (z-\calJ_l) = \sum_{j=0}^{2m}\kappa_j z^j.
		\end{equation}
		Then
		\begin{equation}\label{dist111}
			\left\lvert \det{\frac{\partial{(\kappa_0,\ldots, \kappa_{2m-2})}}{\partial{(\lambda_1,\ldots, \lambda_{m},w_1, \ldots, w_{m-1})}}} \right\rvert=2^m l^{m-1}\prod_{j=1}^m \lambda_j \prod_{1\leq j<k\leq m} |\lambda_j^2-\lambda_k^2|^2.
		\end{equation}
		
		\item 
		Let $\calJ$ be a $(2n+1)\times  (2n+1)$ Jacobi matrix of the form \eqref{generalJ} with $a_1,\ldots,a_{2n}>0$ and $n>0$. Denote $\mu$ to be its spectral measure~\eqref{eq:spMeasDiscrete2}.
		Let 
		\begin{equation}\label{det133}
			\det (z-\calJ_l) = \sum_{j=0}^{2n+1}\kappa_j z^j.
		\end{equation}
		
		Then
		\begin{equation}\label{dist23}
			\left\lvert \det{\frac{\partial{(\kappa_0,\ldots, \kappa_{2n-1})}}{\partial{(\lambda_1,\ldots, \lambda_{n},w_1, \ldots, w_{n})}}} \right\rvert=2^n l^{n}\prod_{j=1}^n \lambda_j^3 \prod_{1\leq j<k\leq n} |\lambda_j^2-\lambda_k^2|^2.
		\end{equation}
	\end{enumerate}	
\end{theorem}

\begin{proof}
	\begin{enumerate}[$(i)$]
		\item Note that $\kappa_{2m}=1$ and $\kappa_{2m-1}=-l$ are fixed constants here.
		
		Let 	$\textbf{m}(z)=\langle e_1, (\calJ-z)^{-1} e_1  \rangle $. Then
		
		\begin{equation}\label{mfun}
			\textbf{m}(z)=\sum_{j=1}^m \frac{w_j}{2}\left(\frac{1}{\lambda_j-z}+\frac{1}{-\lambda_j-z}\right)=z\sum_{j=1}^m\frac{w_j}{\lambda_j^2-z^2}.
		\end{equation}
		
		First, we observe that  		
		\begin{align}\label{det15}
			\sum_{j=0}^{2m}\kappa_j z^j&=\det (z-\calJ_l)\\
			&= \det(z-\calJ) \det (I-(z-\calJ)^{-1} l I_{1\times 1})\\
			&=(1+l \textbf{m}(z)) \prod_{j=1}^{m} (z^2-\lambda_j^2) \label{longl1}
		\end{align}
	and
	\begin{equation}\label{sss}
		l\textbf{m}(z)\prod_{j=1}^m(z^2-\lambda_j^2)=-lz\sum_{j=1}^m w_j \prod_{\substack{1\leq k \leq m\\
				k\neq j}} (z^2-\lambda_k^2).
	\end{equation}
	
		Let
		\begin{align}\label{cjdj}
			&c_j= \kappa_{2j}, \quad  j=0,\ldots,m,\\
			&d_j= \kappa_{2j+1}, \quad j=0,\ldots,m-1,
		\end{align}
	where $c_{m}=1$, $d_{m-1}=-l$.
	
	Letting $u=z^2$ and $\lambda_j^\prime=\lambda_j^2$ we get from 
	\eqref{longl1} and \eqref{sss} that 
			\begin{align}
				&\sum_{j=0}^{m} c_j u^{j}=  \prod_{j=1}^{m} (u-\lambda'_j),\label{cjdj22}\\
				&\sum_{j=0}^{m-1} d_j u^{j}=  -l\sum_{j=1}^m w_j \prod_{\substack{1\leq k \leq m \label{cjdj222}\\
						k\neq j}} (u-\lambda'_k).
			\end{align}
		
		From~\eqref{cjdj22} we get
		\begin{equation}\label{c0}
			\left\lvert \det{\frac{\partial{(c_0,\ldots, c_{m-1})}}{\partial{(\lambda_1^\prime,\ldots, \lambda_{m}^\prime)}}}\right\rvert=\prod_{1\leq  j <k\leq m}|\lambda_j^\prime-\lambda_k^\prime|= \prod_{1\leq  j <k\leq m}|\lambda_j^2-\lambda_k^2|.
		\end{equation}
		Since
		\begin{equation}
			\left\lvert \det{\frac{\partial{(\lambda_1^\prime,\ldots, \lambda_{m}^\prime)}}{\partial{(\lambda_1,\ldots, \lambda_{m})}}}\right\rvert=2^m \prod_{j=1}^m \lambda_j,
		\end{equation}
		\eqref{c0} yields
		\begin{equation}\label{c00}
			\left\lvert \det{\frac{\partial{(c_0,\ldots, c_{m-1})}}{\partial{(\lambda_1,\ldots, \lambda_{m})}}}\right\rvert=2^m \prod_{j=1}^m \lambda_j \prod_{1\leq  j <k\leq m}|\lambda_j^2-\lambda_k^2|.
		\end{equation}
		By \eqref{cjdj22},
		\begin{equation}\label{zero}
			\frac{\partial{(c_0,\ldots, c_{m-1})}}{\partial{(w_1,\ldots, w_{m-1})}}=\begin{pmatrix}\mathbf{0}_{m\times (m-1)} \end{pmatrix}.
		\end{equation}

%
		Now we consider~\eqref{cjdj222}. 
		In view of \cite[eq.(5.9), eq.(5.14)]{Koz17}, \eqref{cjdj222} implies that
		\begin{align}
			\left\lvert \det{\frac{\partial{(d_0,\ldots, d_{m-2})}}{\partial{(w_1,\ldots, w_{m-1})}}}\right\rvert&=l^{m-1} \prod_{1\leq j<k\leq m}|\lambda_j^\prime-\lambda_k^\prime|\\
			&=l^{m-1} \prod_{1\leq j<k\leq m}|\lambda_j^2-\lambda_k^2|.\label{add}
		\end{align}		
		Combining \eqref{c00}, \eqref{zero}, \eqref{add} we get
		\begin{align}
			\left\lvert \det{\frac{\partial{(\kappa_0,\ldots, \kappa_{m-2})}}{\partial{(\lambda_1,\ldots, \lambda_{m}, w_1,\ldots, w_{m-1})}}}\right\rvert&=\left\lvert \det{\frac{\partial{(c_0,\ldots, c_{m-1}, d_0, \ldots, d_{m-2})}}{\partial{(\lambda_1,\ldots, \lambda_{m}, w_1,\ldots, w_{m-1})}}}\right\rvert\\
			&=2^m l^{m-1} \prod_{j=1}^m \lambda_j \prod_{1\leq j<k\leq m}|\lambda_j^2-\lambda_k^2|^2.\label{resu1}
		\end{align}		
		
		\item Note that $\kappa_{2n+1}=1$ and $\kappa_{2n}=-l$ are constants.  We again start with 	$\textbf{m}(z)=\langle e_1, (\calJ-z)^{-1} e_1  \rangle $ which becomes
		 \begin{equation}\label{mfun1}
		 	\textbf{m}(z)=\sum_{j=1}^n \frac{w_j}{2}\left(\frac{1}{\lambda_j-z}+\frac{1}{-\lambda_j-z}\right)-\frac{w_0}{z}=z\sum_{j=1}^n\frac{w_j}{\lambda_j^2-z^2}-\frac{w_0}{z}.
		 \end{equation}
	 Now,
		\begin{align}\label{det11}
			\sum_{j=0}^{2n+1}\kappa_j z^j&=\det (z-\calJ_l)\\
			&= \det(z-\calJ) \det (I-(z-\calJ)^{-1} l I_{1\times 1})\\
			&=(1+l\textbf{m}(z)) z\prod_{j=1}^{n} (z^2-\lambda_j^2) \label{longl12}
		\end{align}
	and
	\begin{align}\label{sss2}
		lz\textbf{m}(z)\prod_{j=1}^n(z^2-\lambda_j^2)=-l\sum_{j=0}^n w_j \prod_{\substack{0\leq k \leq n.\\
				k\neq j}} (z^2-\lambda_k^2)
	\end{align}

		Define
		\begin{align}\label{cjdj3}
			c_j&= \kappa_{2j+1},\,\,\,\,\, j=0,\ldots,n,\\
			d_j&= \kappa_{2j},\,\,\,\,\, j=0,\ldots,n,
		\end{align}
	with $c_n=1$, $d_n=-l$.
	Taking $u=z^2$ and $\lambda_j^\prime=\lambda_j^2$ we get from \eqref{longl12} and \eqref{sss2} that 
		\begin{align}
			\label{c22}
			&\sum_{j=0}^{n} c_j u^{j}=  \prod_{j=1}^{n} (u-\lambda'_j), \\
			\label{eq:dpoly}
			&\sum_{j=0}^{n} d_j u^{j}=  -l\sum_{j=0}^n w_j \prod_{\substack{0\leq k \leq n\\
					k\neq j}} (u-\lambda'_k).
		\end{align}

		Using~\eqref{c22} we  get
		\begin{equation}\label{c1}
			\left\lvert \det{\frac{\partial{(c_0,\ldots, c_{n-1})}}{\partial{(\lambda_1,\ldots, \lambda_{n})}}}\right\rvert=2^n\prod_{j=1}^n\lambda_j \prod_{1\leq  j <k\leq n}|\lambda_j^2-\lambda_k^2|.
		\end{equation}
		Using~\eqref{eq:dpoly},
		\begin{align}
			\left\lvert \det{\frac{\partial{(d_0,\ldots, d_{n-1})}}{\partial{(w_1,\ldots, w_{n})}}}\right\rvert&=l^{n} \prod_{0\leq j<k\leq n}|\lambda_j^\prime-\lambda_k^\prime|\\
			&=l^{n} \prod_{j=1}^n \lambda_j^2 \prod_{1\leq j<k\leq n}|\lambda_j^2-\lambda_k^2|.\label{add2}
		\end{align}		
		Combining \eqref{c1}, \eqref{add2} we get
		\begin{align}
			\left\lvert \det{\frac{\partial{(\kappa_0,\ldots, \kappa_{2n-1})}}{\partial{(\lambda_1,\ldots, \lambda_{n}, w_1,\ldots, w_{n})}}}\right\rvert&=\left\lvert \det{\frac{\partial{(c_0,\ldots, c_{n-1}, d_0, \ldots, d_{n-1})}}{\partial{(\lambda_1,\ldots, \lambda_{n}, w_1,\ldots, w_{n})}}}\right\rvert\\
			&=2^n l^{n} \prod_{j=1}^n \lambda_j^3 \prod_{1\leq j<k\leq n}|\lambda_j^2-\lambda_k^2|^2.\label{trtr11}
		\end{align}	
	\end{enumerate}	
\end{proof}

Notice that in the case~\eqref{eq:Jl} coefficients of $\kappa$ were real, while in the case~\eqref{eq:Jil} they are real or purely imaginary.
Indeed, for a monic polynomial $\kappa(z)=\sum_{j=0}^k \kappa_j z^j$ of degree $k$ whose zeros are symmetric with respect to imaginary axis, 
$$
Q(z)=i^{k}\kappa(z/i)
$$
 is a monic polynomial with real coefficients. This means 
$\kappa(z)=Q(iz)i^{-k}$, and therefore 
$\Im \kappa_{k-2}=\Im \kappa_{k-4}=\dots=0$ and $\Re \kappa_{k-1}=\Re \kappa_{k-3}=\dots =0$. 

\begin{theorem}\label{eigd}
	Let $l>0$.
	\begin{enumerate}[$(i)$]
		\item Let $\calJ$ be a $2m\times 2m$ Jacobi matrix of the form \eqref{generalJ} with $a_1,\ldots,a_{2m-1}>0$ and $m>0$. Denote $\mu$ to be its spectral measure~\eqref{eq:spMeasDiscrete}.
		Let 
		\begin{equation}\label{det1}
			\det (z-\calJ_{il}) = \sum_{j=0}^{2m}\kappa_j z^j.
		\end{equation}
		Then
		\begin{multline}\label{dist1}
			\left\lvert \det{\frac{\partial{(\Re\kappa_0,\Im\kappa_1,\ldots,\Re\kappa_{2m-4},\Im\kappa_{2m-3},\Re\kappa_{2m-2})}}{\partial{(\lambda_1,\ldots, \lambda_{m},w_1, \ldots, w_{m-1})}}} \right\rvert
			\\ =2^m l^{m-1}\prod_{j=1}^m \lambda_j \prod_{1\leq j<k\leq m} |\lambda_j^2-\lambda_k^2|^2.
		\end{multline}
		
		\item 
		Let $\calJ$ be a $(2n+1)\times  (2n+1)$ Jacobi matrix of the form \eqref{generalJ} with $a_1,\ldots,a_{2n}>0$ and $n>0$. Denote $\mu$ to be its spectral measure~\eqref{eq:spMeasDiscrete2}.
		Let 
		\begin{equation}\label{det13}
			\sum_{j=0}^{2n+1}\kappa_j z^j=\det (z-\calJ_{il}).
		\end{equation}
		
		Then
		\begin{equation}\label{dist2}
			\left\lvert \det{\frac{\partial{(\Im\kappa_0,\Re\kappa_1,\ldots, \Im\kappa_{2n-2},\Re\kappa_{2n-1})}}{\partial{(\lambda_1,\ldots, \lambda_{n},w_1, \ldots, w_{n})}}} \right\rvert=2^n l^{n}\prod_{j=1}^n \lambda_j^3 \prod_{1\leq j<k\leq n} |\lambda_j^2-\lambda_k^2|^2.
		\end{equation}
	\end{enumerate}	
	
\end{theorem}
\begin{proof}
	The only difference from the setting in the previous theorem is that $l$ gets an extra factor of $i$, and the same happens with the coefficients $\kappa_{2j-1}$'s in (i) or $\kappa_{2j}$'s in (ii). The modulus of the Jacobian in~\eqref{dist1} and~\eqref{dist2} is therefore the same as in~\eqref{dist111} and~\eqref{dist23}, respectively.
\end{proof}


\section{Eigenvalues for rank one Hermitian perturbations}\label{ss:Hermitian}

\begin{theorem}\label{th:Hermitian}
	Let $\calJ$ belong to chG$\beta$E  $($see Section~\ref{ss:chiral}$)$, $l>0$, $a=|n-m|+1-2/\beta$, and
	\begin{equation}\label{eq:Jl2}
		\calJ_l:=\calJ+lI_{1\times 1}.
	\end{equation}
	\begin{enumerate}[$(i)$]
		\item Let $m\le n$. The eigenvalues of $\calJ_l$ are distributed on 
		\begin{align}
			\left\{(z_j)_{j=1}^{2m}: \sum_{j=1}^{2m} z_j=l,\,\,\,  \,\, z_1> -z_2>z_3>\cdots >z_{2m-1}>-z_{2m} >0 \right\}
		\end{align}
		according to 
\begin{align}
	 \frac{1}{Z_{\beta,m,a}} \, \,l^{1-\frac{m\beta}2}e^{l^2/4} \, {\prod_{j=1}^{2m}}|z_j|^{\frac{2\beta a -\beta+2}{4}}e^{-z_j^2/4}\prod_{1\leq j<k\leq 2m}|z_j-z_k|
	\prod_{j,k=1}^{2m} |z_j+z_k|^{\frac{\beta-2}{4}}
	\prod_{j=1}^{2m-1} dz_j.
	\label{real1}
\end{align}
Here
\begin{equation}\label{eq:Z}
Z_{\beta,m,a} = \frac{2^{m(\beta-2)/2} \,h_{\beta,m,a} \,[\Gamma(\beta/2)]^m}{m!  \Gamma(\beta m /2)}.
\end{equation}
		\item Let $m\ge n+1$.  The eigenvalues of $\calJ_l$ are distributed on 
		\begin{align}
			\left\{(z_j)_{j=1}^{2n+1}: \sum_{j=1}^{2n+1} z_j=l,\,\,\,  \,\,z_1> -z_2>z_3>\cdots >-z_{2n}>z_{2n+1} >0 \right\}
		\end{align}
		according to 
\begin{align}
	\frac{l^{1-\frac{m\beta}2}e^{l^2/4}}{W_{\beta,m,n,a}} \, \prod_{j=1}^{2n+1}|z_j|^{\frac{2\beta m -2\beta n-\beta-2}{4}}e^{-z_j^2/4}\prod_{1\leq j<k\leq 2n+1}|z_j-z_k| \prod_{j,k=1}^{2n+1} |{z_j}+z_k|^{\frac{\beta-2}{4}} \prod_{j=1}^{2n} dz_j. \label{real2}
\end{align}
	\end{enumerate}
	Here
	\begin{equation}\label{eq:W}
	W_{\beta,m,n,a}=\frac{2^{\frac{(2n+1)(\beta-2)}{4}} \,h_{\beta,n,a} \,[\Gamma(\beta/2)]^n\,\Gamma(\beta(m-n)/2) }{n! \Gamma(\beta m /2)  }.
	\end{equation}
\end{theorem}
\begin{remarks}
	1.  1. As a corollary, eigenvalues of  Hermitian perturbations of chGOE, chGUE, chGSE (see Proposition~\ref{pr:Model1}) are ~\eqref{real1} together with $z=0$ of algebraic multiplicity $n-m$ (for the case $m\le n$), and ~\eqref{real2} together with $z=0$ of algebraic multiplicity $m-n-1$ (for the case $m\ge n+1$).
	
	2. See the end of this section for the case when $l$ is not deterministic but random.
\end{remarks}

\begin{proof}
	\begin{enumerate}[$(i)$]
		\item Let $\sum_{j=0}^{2m} \kappa_j z^j= \det(z-\calJ_l).$ Then
		
		\begin{align}
			\prod_{1\leq j<k\leq 2m}|z_j-z_k|&=	\left\lvert \det{\frac{\partial{(\kappa_0,\ldots, \kappa_{2m-1})}}{\partial{(z_1,\ldots, z_{2m})}}}\right\rvert \label{a1}\\
			&=	\left\lvert \det{\frac{\partial{(\kappa_0,\ldots, \kappa_{2m-1})}}{\partial{(z_1,\ldots,z_{2m-1}, \kappa_{2m-1)}}}}\right\rvert \label{a2}\\
			&=\left\lvert \det{\frac{\partial{(\kappa_0,\ldots, \kappa_{2m-2})}}{\partial{(z_1,\ldots, z_{2m-1})}}}\right\rvert\label{a3}.
		\end{align}		
		The equality \eqref{a1} is well known, \eqref{a2} is a result of $\sum_{j=1}^{2m}z_j=-\kappa_{2m-1}$
		 and \eqref{a3} follows by removing the last row and column from the determinant~\eqref{a2}.

		Combining  part $(i)$ of  Theorem \ref{thmmm}, \eqref{a3} and \eqref{dist111} we get the density of $dz_1\cdots dz_{2m-1}$:
		\begin{align}
			m! \frac{l^{1-m}\displaystyle\prod_{1\leq j<k\leq 2m}|z_j-z_k|}{h_{\beta,m,a}\displaystyle\prod_{1\leq j<k\leq m}|\lambda_j^2-\lambda_k^2|^2}&\prod_{j=1}^m\lambda_j^{\beta a}e^{-\lambda_j^2/2}\prod_{1\leq j<k\leq m}|\lambda_k^2-\lambda_j^2|^\beta   \nonumber     \\
			&\times  \Gamma(\beta m/2)\prod_{j=1}^m \frac{w_j^{\beta/2-1}}{\Gamma(\beta/2)}\label{bigdist22}.
		\end{align}
		Notice the extra factor of $m!$ that comes from the fact that $\lambda_j$'s were not ordered while $z_j$'s are. 
		
		It follows from \eqref{det15}, \eqref{cjdj22} that 
		\begin{equation}\label{yt1}
			\sum_{j=1}^m \lambda_j^2=-c_{m-1}=-\kappa_{2m-2}=-\sum_{1\leq i<j\leq 2m}z_i z_j. 
		\end{equation}
		Since $\sum_{j=1}^{2m} z_j=l$, we have 
		\begin{align}
			l^2&= \sum_{j=1}^{2m} z_j^2+2\sum_{{1\leq i<j\leq 2m}}z_i z_j \label{yt2}\\
			&=\sum_{j=1}^{2m} z_j^2-2\sum_{j=1}^{m} \lambda_j^2.\label{yt3}
		\end{align}
		Thus
		\begin{equation}\label{pp1}
			\sum_{j=1}^{m} \lambda_j^2=\frac{-l^2+\sum_{j=1}^{2m} z_j^2}{2}.
		\end{equation}
		It follows from \eqref{det15}, \eqref{cjdj22} that 
		\begin{equation}\label{pp2}
			\prod_{j=1}^m \lambda_j^2= |c_0| = |\kappa_0|= \prod_{j=1}^{2m} |z_j| .
		\end{equation}
		By \eqref{longl1}, we have
		\begin{align}\label{wjs}
			\displaystyle	 \frac{w_j}{2}=\left\lvert  \mathrm{Res}_{z=\lambda_j} \textbf{m}(z) \right\rvert =\left\lvert \mathrm{Res}_{z=\lambda_j}\frac{\prod_{k=1}^{2m} (z-z_k)}{l\prod_{k=1}^{m}       
				(z^2-\lambda_k^2)}\right\rvert=\left\lvert  \frac{\prod_{k=1}^{2m} (\lambda_j-z_k)}{2l\lambda_j \prod_{\substack{1\leq k \leq m\\
						k\neq j}}  (\lambda_k^2-\lambda_j^2) }\right\rvert.
		\end{align}
		Similarly,
		\begin{align}\label{wjs2}
			\displaystyle	 \frac{w_j}{2}=\left\lvert  \mathrm{Res}_{z=-\lambda_j} \textbf{m}(z) \right\rvert =\left\lvert \mathrm{Res}_{z=-\lambda_j}\frac{\prod_{k=1}^{2m} (z-z_k)}{l\prod_{k=1}^{m}       
				(z^2-\lambda_k^2)}\right\rvert=\left\lvert  \frac{\prod_{k=1}^{2m} (\lambda_j+z_k)}{2l\lambda_j \prod_{\substack{1\leq k \leq m\\
						k\neq j}}  (\lambda_k^2-\lambda_j^2) }\right\rvert.
		\end{align}
		
		By \eqref{cjdj22} 
		\begin{equation}\label{split}
			\prod_{k=1}^m (z^2-\lambda_k^2)=\sum_{j=0}^m \kappa_{2j}z^{2j}=\frac{1}{2}\prod_{k=1}^{2m}(z-z_k)+\frac{1}{2}\prod_{k=1}^{2m}(z+z_k)
		\end{equation}
		is satisfied.

		Letting $z=z_1,\ldots,z_{2m}$ in \eqref{split} yields
		\begin{equation}\label{sum}
				\prod_{\substack{k=1,\ldots, 2m\\
					j=1,\ldots, m}}  |z_k^2-\lambda_j^2| =		\frac{1}{4^m}\prod_{k,j=1}^{2m}  |z_j+z_k| .
		\end{equation}
		Combining \eqref{wjs}, \eqref{wjs2} , and \eqref{sum}, and  we get
		
		\begin{equation}\label{summ}
			\prod_{j=1}^{m} w_j^2= 
				 \frac{   \prod_{k,j=1}^{2m}  |z_j+z_k| }{l^{2m} 4^m \prod_{j=1}^m  \lambda_j^2   \prod_{1\leq  j<k\leq m} |\lambda_k^2-\lambda_j^2 |^4  } 
		\end{equation}
		Substituting \eqref{summ}, \eqref{pp1}, \eqref{pp2} into \eqref{bigdist22} we obtain \eqref{real1}.

		\item 
		Let $\sum_{j=0}^{2n+1} \kappa_j z^j=\det(z-\calJ_l).$   By a similar argument as in (i), we see that 
		
		\begin{align}
			\prod_{1\leq j<k\leq 2n+1}|z_j-z_k|=\left\lvert \det{\frac{\partial{(\kappa_0,\ldots, \kappa_{2n-1})}}{\partial{(z_1,\ldots, z_{2n})}}}\right\rvert\label{b1}
		\end{align}		
		and
		\begin{equation}\label{rr1}
			\sum_{j=1}^{n} \lambda_j^2 =\frac{-l^2+\sum_{j=1}^{2n+1} z_j^2}{2}.
		\end{equation}
		
		Using  part $(ii)$ in Theorem \ref{thmmm}, \eqref{b1} and \eqref{dist23}, we find the distribution of the $z_j$'s:
		
		\begin{align}
			&n! \frac{\prod_{1\leq j<k\leq 2n+1}|z_j-z_k|}{2^n l^{n}\prod_{j=1}^n \lambda_j^3 \prod_{1\leq j<k\leq n} |\lambda_j^2-\lambda_k^2|^2}\frac{2^n\prod_{j=1}^n\lambda_j}{h_{\beta,n,a}}\prod_{j=1}^n\lambda_j^{\beta a}e^{-\lambda_j^2/2} \nonumber\\
			&\times \prod_{1\leq j<k\leq n}|\lambda_k^2-\lambda_j^2|^\beta    \times  \frac{w_0^{\beta(m-n)/2-1}}{\Gamma{(\beta(m-n)/2)}} \times\Gamma(\beta m/2)\prod_{j=1}^n\frac{w_j^{\beta/2-1}}{\Gamma(\beta/2)}\nonumber\\
			& \times  dz_1\cdots dz_{2n}. \label{bbb9}
		\end{align}

		It follows from \eqref{longl12} that
		\begin{align}\label{tunn}
			\displaystyle w_0=\left\lvert  \mathrm{Res}_{z=0} \textbf{m}(z) \right\rvert =\left\lvert \mathrm{Res}_{z=0}\frac{\prod_{k=1}^{2n+1} (z-z_k)}{lz\prod_{k=1}^{n}       
				(z^2-\lambda_k^2)}\right\rvert=\left\lvert  \frac{\prod_{k=1}^{2n+1} z_k}{l \prod_{k=1}^n \lambda_k^2}   \right\rvert.
		\end{align}
		
		Similarly,
		\begin{align}\label{las2}
			\displaystyle	\frac{w_j}{2}=\left\lvert  \mathrm{Res}_{z=\lambda_j} \textbf{m}(z) \right\rvert =\left\lvert \mathrm{Res}_{z=\lambda_j}\frac{\prod_{k=1}^{2n+1} (z-z_k)}{lz\prod_{k=1}^{n}       
				(z^2-\lambda_k^2)}\right\rvert&=\left\lvert  \frac{\prod_{k=1}^{2n+1} (\lambda_j-z_k)}{2l\lambda_j^2 \prod_{\substack{1\leq k \leq n\\
						k\neq j}}  (\lambda_k^2-\lambda_j^2) }\right\rvert\\
		\label{las3}
		=\left\lvert  \mathrm{Res}_{z=-\lambda_j} \textbf{m}(z) \right\rvert &=\left\lvert  \frac{\prod_{k=1}^{2n+1} (\lambda_j+z_k)}{2l\lambda_j^2 \prod_{\substack{1\leq k \leq n\\
						k\neq j}}  (\lambda_k^2-\lambda_j^2) }\right\rvert		
					.
		\end{align}
		By \eqref{longl12} and \eqref{sss2}
		
		\begin{equation}\label{las1}
			z	\prod_{k=1}^n (z^2-\lambda_k^2)=\frac{1}{2}\prod_{k=1}^{2n+1}(z-z_k)+\frac{1}{2}\prod_{k=1}^{2n+1}(z+z_k).
		\end{equation}
		Letting $z=z_1,\ldots,z_{2n+1}$ in \eqref{las1} implies
		
		\begin{equation}\label{tum}
				\prod_{\substack{k=1,\ldots, 2n+1\\
					j=1,\ldots, n}}  |z_k^2-\lambda_j^2| =		\frac{\prod_{k,j=1}^{2n+1} |z_j+z_k|}{2^{2n+1}\prod_{k=1}^{2n+1} |z_k|} .
		\end{equation}
		Combining \eqref{tum}, \eqref{las2}, \eqref{las3}, we obtain
		
		\begin{equation}\label{tumm}
			\prod_{j=1}^{n} w_j^2= 
				 \frac{   \prod_{k,j=1}^{2n+1}  |z_j+z_k| }{l^{2n} 2^{2n+1} \prod_{j=1}^{2n+1} |z_j| \prod_{j=1}^n  \lambda_j^4   \prod_{1\leq  j<k\leq n} |\lambda_k^2-\lambda_j^2 |^4  }.
		\end{equation}
		Substituting \eqref{rr1}, \eqref{tunn}, \eqref{tumm} into \eqref{bbb9}, we get \eqref{real2}.
		
	\end{enumerate}
\end{proof}
It is natural to choose $l$ to be random and independent of $\calJ$. For example, let $l$ be $\sqrt{2} \chi_{\beta m/2}$-distributed, i.e., with probability distribution
$$
F(l) \, dl= \frac{1}{2^{\beta m/2-1}\Gamma(\beta m/4)} l^{\frac{m\beta}{2}-1} e^{-l^2/4} \, dl
$$
 on $(0,\infty)$. Then making an extra change of variables from $\{z_1,\ldots,z_{k-1},l\}$ to $\{z_1,\ldots,z_{k}\}$, we arrive at the following joint distribution of eigenvalues:
 \begin{itemize}
 	\item[(i)] If $m\le n$, then eigenvalues of $\calJ_l$ are distributed on
 			\begin{equation}\label{eq:confSpace}
 				\left\{(z_j)_{j=1}^{2m}: \,\,\,  \,\, z_1> -z_2>z_3>\cdots >z_{2m-1}>-z_{2m} >0 \right\}
 			\end{equation}
 			according to 
 			\begin{equation}
 				\frac{1}{\tilde{Z}_{\beta,m,a}} \, \,\, {\prod_{j=1}^{2m}}|z_j|^{\frac{2\beta a -\beta+2}{4}}e^{-z_j^2/4}\prod_{1\leq j<k\leq 2m}|z_j-z_k|
 				\prod_{j,k=1}^{2m} |z_j+z_k|^{\frac{\beta-2}{4}}
 				\prod_{j=1}^{2m} dz_j.
 			\end{equation}
 			Here
 			\begin{equation}
 				\tilde{Z}_{\beta,m,a} = \frac{2^{m \beta-m-1} \,h_{\beta,m,a} \Gamma(\beta m/4)\,[\Gamma(\beta/2)]^m}{m!  \Gamma(\beta m /2)}.
 			\end{equation}
 			For $\beta=2$ this takes an especially simple form
 			\begin{equation}\label{eq:notPfaff}
 				\frac{1}{\tilde{Z}_{2,m,|n-m|}} \, \,\, {\prod_{j=1}^{2m}}|z_j|^{|n-m|}e^{-z_j^2/4}\prod_{1\leq j<k\leq 2m}|z_j-z_k|
 				\prod_{j=1}^{2m} dz_j.
 			\end{equation}
 		At first sight one might expect  that~\eqref{eq:notPfaff} has a Pfaffian structure but recall the configuration space is~\eqref{eq:confSpace} which complicates analysis substantially.
 		
 		\item[(ii)] If $m\ge n+1$, then eigenvalues of $\calJ_l$ are distributed on 
 		\begin{align}
 			\left\{(z_j)_{j=1}^{2n+1}: \,\,\,  \,\,z_1> -z_2>z_3>\cdots >-z_{2n}>z_{2n+1} >0 \right\}
 		\end{align}
 		according to 
 		\begin{align}
 			\frac{1}{\tilde{W}_{\beta,m,n,a}} \, \prod_{j=1}^{2n+1}|z_j|^{\frac{2\beta m -2\beta n-\beta-2}{4}}e^{-z_j^2/4}\prod_{1\leq j<k\leq 2n+1}|z_j-z_k| \prod_{j,k=1}^{2n+1} |{z_j}+z_k|^{\frac{\beta-2}{4}} \prod_{j=1}^{2n+1} dz_j. 
 		\end{align}
 Here
\begin{equation}
\tilde{W}_{\beta,m,n,a}=\frac{2^{\frac{(2n+1)(\beta-2)}{4}+\frac{\beta m}{2}-1} \,h_{\beta,n,a} \,\Gamma(\beta m/4) [\Gamma(\beta/2)]^n\,\Gamma(\beta(m-n)/2) }{n! \Gamma(\beta m /2)  }.
\end{equation}
For $\beta=2$ this becomes
\begin{align}
	\frac{1}{\tilde{W}_{2,m,n,|m-n|}} \, \prod_{j=1}^{2n+1}|z_j|^{ m -n-1}e^{-z_j^2/4}\prod_{1\leq j<k\leq 2n+1}|z_j-z_k|  \prod_{j=1}^{2n+1} dz_j. 
\end{align}
 \end{itemize}


\section{Eigenvalues for rank one non-Hermitian perturbations}\label{ss:nonHermitian}

Let $\calJ$ be an $N\times N$ random matrix from chG$\beta$E, and consider
\begin{equation}\label{eq:Jil2}
	\calJ_{il}:=\calJ+ilI_{1\times 1} 
\end{equation}
for some $l>0$.

In order to simplify the final answer we will assume $l$ to be random, independent from $\calJ$ (or $H$ for $\beta=1,2,4$) with absolutely continuous distribution $F(l) \,dl$ with $F(l)>0$ for $l>0$ and $0$ otherwise. Other distributions of $l$ (or the deterministic case) can also be treated in the exact same manner, and we  leave it as an exercise to an interested reader.

As we discussed in Proposition~\ref{bijection}, eigenvalues of~\eqref{eq:Jil2}  belong to $\mathbb{C}_+$, and they are symmetric with respect to the imaginary axis. The set of all possible configurations $\{z_j\}_{j=1}^N$ of these eigenvalues, therefore, decomposes as the disjoint union
$$
X_N:=\bigcup_{\stackrel{L\ge 0, M\ge0}{L+2M=N}} X_{L,M},
$$
where
\begin{multline}
X_{L,M}:=
\Big\{
\{z_j\}_{j=1}^N \in\mathbb{C}_+^N:  z_1,\ldots,z_L \in i\bbR_+; \\
 z_{L+1} = - \bar{z}_{L+1+M}, \ldots, z_{L+M} = -\bar{z}_{L+2M}
\Big\}.
\end{multline}
For each $z_j$, let $z_j=x_j + i y_j$, $x_j,y_j\in\bbR$.

We will say that $\{z_j\}_{j=1}^N$ on $X_N$ have joint distribution $f(z_1,\ldots,z_N) \big| \bigwedge_{j=1}^N z_j \big|$ (with $f$ being invariant under permutation of its arguments), if conditionally on the event $\{z_j\}_{j=1}^N \in X_{L,M}$ the distribution becomes 
\begin{multline}\label{eq:wedges}
	2^M \frac{1}{ L! M!2^M} f(iy_1,\ldots,iy_L,\pm x_{L+1}+iy_{L+1},\pm x_{L+2}-i y_{L+2},\ldots,\pm x_{L+M}+i y_{L+M})
	\\
	\times  \prod_{j=1}^L \,dy_j  \prod_{j=L+1}^{L+M} (dx_{j} dy_{j}) .
\end{multline}
Here the factor $\tfrac{1}{ L! M!2^M}$ corresponds to the number of permutations on $X_{L,M}$ that preserve the configuration, and $2^M$ comes from $\big| dz \wedge d(-\bar{z}) \big| = 2 dx\,dy$.

For a more formal introduction to such point processes, we refer the reader to \cite{boro}.

%
%
%

\begin{theorem}\label{th:nonHermitian}
	Let $\calJ$ belong to chG$\beta$E  $($see Section~\ref{ss:chiral}$)$, $a=|n-m|+1-2/\beta$, 
	\begin{equation}
		\calJ_{il}:=\calJ+ilI_{1\times 1},
	\end{equation}
where $l$ is independent of $\calJ$ with distribution $F(l)dl$, $F(l)>0$ for $l>0$ and $0$ otherwise.
	\begin{enumerate}[$(i)$]
		\item Let $m\le n$.
		Then $\{z_j\}_{j=1}^{2m}$ are jointly distributed on $X_{2m}$ according to
\begin{multline}
	\frac{1}{Z_{\beta,m,a}} \,F(l)  l^{1-\frac{m\beta}2} e^{-l^2/4} {\prod_{j=1}^{2m}}|z_j|^{\frac{2\beta a -\beta+2}{4}}e^{-z^2_j/4}
	\\
	\times \prod_{1\leq j<k\leq 2m}|z_j-z_k|
	  \prod_{j,k=1}^{2m} |{z_j}-\bar{z}_k|^{\frac{\beta-2}{4}} \, \Big|\bigwedge_{j=1}^{2m} dz_j \Big|,
	\label{www1}
\end{multline}
	where $l=\big|\sum_{j=1}^{2m} z_j\big|$ and $Z_{\beta,m,a}$ is~\eqref{eq:Z}.
		\item Let $m\ge n+1$. Then $\{z_j\}_{j=1}^{2n+1}$ are jointly distributed on $X_{2n+1}$ according to
\begin{multline}
	 \frac{1}{W_{\beta,m,n,a}} \, F(l) l^{1-\frac{m\beta}2} e^{-l^2/4}  \prod_{j=1}^{2n+1}|z_j|^{\frac{2\beta m -2\beta n-\beta-2}{4}}e^{-z_j^2/4}	  
	\\
	\,\times \prod_{1\leq j<k\leq 2n+1}|z_j-z_k|
	\prod_{j,k=1}^{2n+1} |{z_j}-\bar{z}_k|^{\frac{\beta-2}{4}}   \Big|\bigwedge_{j=1}^{2n+1} dz_j \Big|,
	\label{www2}
\end{multline}
	 where $l=\big|\sum_{j=1}^{2n+1} z_j\big|$ and $W_{\beta,m,n,a}$ is~\eqref{eq:W}.
	\end{enumerate}
	
\end{theorem}
\begin{remarks}
	1. As a corollary, eigenvalues of  non-Hermitian perturbations of chGOE, chGUE, chGSE (see Proposition~\ref{pr:Model2}) are ~\eqref{www1} together with $z=0$ of algebraic multiplicity $n-m$ (for the case $m\le n$), and ~\eqref{www2} together with $z=0$ of algebraic multiplicity $m-n-1$ (for the case $m\ge n+1$).
	
	2. Even though $z_j$'s are in $\bbC_+$, because of the symmetry $\sum z_j^2 = \sum \Re(z_j^2)$ is a real quantity.
	
\end{remarks}
\begin{proof}
	\begin{enumerate}[$(i)$]
		\item  
		Recall the characteristic polynomial $\kappa(z)$ in~\eqref{det1} and that 
		$$
		Q(z)=i^{N}\kappa(z/i)
		$$
		is a monic polynomial with real coefficients and zeros at $\{iz_j\}_{j=1}^{2m}$. 
		
		Let us assume that $z_j$'s belong to $X_{L,M} \subset X_{2m}$.	Using~\cite[Lemma 6.5]{KK} (if one applies it to $Q$), we get 
		\begin{align}\label{vande}
			\left\lvert \det{\frac{\partial{(\Re\kappa_0,\Im\kappa_1,\ldots,\Im\kappa_{2m-3},\Re\kappa_{2m-2},\Im\kappa_{2m-1})}}{\partial{(y_1,\ldots, y_L ,x_{L+1},y_{L+1},\ldots, x_{L+M},y_{L+M})}}}\right\rvert=2^M \prod_{1\leq j<k\leq 2m}|z_j-z_k|.
		\end{align}		
		
		
		Combining \eqref{vande} with  \eqref{dist1} and $\kappa_{2m-1}=-il$ we obtain
		\begin{align}
			\left\lvert \det{\frac{\partial{(\lambda_1,\ldots, \lambda_{m},w_1, \ldots, w_{m-1},l)}} {\partial{(y_1,\ldots, y_L ,x_{L+1},y_{L+1},\ldots, x_{L+M},y_{L+M})}}   } \right\rvert
			\\ = 2^{M-m} \frac{\prod_{1\leq j<k\leq 2m}|z_j-z_k|}{ l^{m-1}\prod_{j=1}^m \lambda_j \prod_{1\leq j<k\leq m} |\lambda_j^2-\lambda_k^2|^2}.
		\end{align}	
	
		Now we use this Jacobian together with Theorem \ref{thmmm}(i) we obtain the  joint density of $x_j$'s and $y_j$'s:
		\begin{multline}\label{eq:intermediateDensity}
			\frac{m!}{2^M M!L!}\frac{2^M}{h_{\beta,m,a}} \prod_{1\leq j<k\leq 2m}|z_j-z_k|
			\prod_{j=1}^m\lambda_j^{\beta a}e^{-\lambda_j^2/2}\prod_{1\leq j<k\leq m}|\lambda_k^2-\lambda_j^2|^{\beta-2}           
			\\
			\times  \Gamma(\beta m/2)\prod_{j=1}^m \frac{w_j^{\beta/2-1}}{\Gamma(\beta/2)} l^{1-m}F(l) 
			\times  \prod_{j=1}^L \,dy_j  \prod_{j=L+1}^{L+M} (dx_{j} dy_{j}). 
		\end{multline}
	Notice the extra factor of $\frac{1}{2^M M!L!}$ since we do not impose ordering on our $z_j$'s so each configuration appears ${2^M M!L!}$ times. Similarly, $m!$ comes from the absence of ordering in $\lambda_j$'s.
		
	Using~\eqref{longl1},
\begin{equation}\label{mfun2}
	\sum_{j=0}^{2m} \kappa_j z^j=(1+il \textbf{m}(z))\prod_{j=1}^m (z^2-\lambda_j^2).
\end{equation}
Here
\begin{align}
	\kappa_{2m}&=1,\\
	\Im \kappa_{2j}&=0,\,\,\,\,\,\, j=0,\ldots,m-1,\\
	\Re \kappa_{2j+1}&=0,\,\,\,\,\,\, j=0,\ldots,m-1,\label{re}
\end{align}
and
\begin{equation}\label{re1}
	\sum_{j=0}^{2m} \Re \kappa_j z^j=\prod_{j=1}^m (z^2-\lambda_j^2).
\end{equation}

	It follows  that 
	\begin{equation}\label{rpp2}
		\prod_{j=1}^m \lambda_j^2= |\Re \kappa_0| = |\kappa_0|= \prod_{j=1}^{2m} |z_j|.
	\end{equation}
and
\begin{equation}\label{yt1}
	\sum_{j=1}^m \lambda_j^2=-\kappa_{2m-2}=-\sum_{1\leq i<j\leq 2m}z_i z_j. 
\end{equation}
Since $\sum_{j=1}^{2m} z_j=\operatorname{Tr}(\calJ_{il}) = il$, we have 
\begin{align}
	-l^2&= \sum_{j=1}^{2m} z_j^2+2\sum_{{1\leq i<j\leq 2m}}z_i z_j. \label{zt2}\\
	&=\sum_{j=1}^{2m} z_j^2-2\sum_{j=1}^{m} \lambda_j^2.\label{zt3}
\end{align}
Thus
\begin{equation}\label{ppp1}
	\sum_{j=1}^{m} \lambda_j^2=\frac{l^2+\sum_{j=1}^{2m} z_j^2}{2}.
\end{equation}

Using \eqref{mfun2}, \eqref{re1}, we obtain
\begin{align}\label{arg}
	\frac{1}{2}\prod_{j=1}^{2m}(z-z_j)+\frac{1}{2}\prod_{j=1}^{2m}(z-\bar{z}_j)=\prod_{j=1}^m (z^2-\lambda_j^2).
\end{align}
Letting $z=z_1,\ldots, z_{2m}$ in \eqref{arg}, we get
\begin{equation}\label{summmm}
	\prod_{\substack{k=1,\ldots, 2m\\
			j=1,\ldots, m}}  |z_k^2-\lambda_j^2| =		\frac{1}{4^m}\prod_{k,j=1}^{2m}  |{z_j}-\bar{z}_k|.
\end{equation}
By~\eqref{mfun2},
\begin{align}\label{wjsw}
	\displaystyle	 \frac{w_j}{2}=\left\lvert  \mathrm{Res}_{z=\lambda_j} \textbf{m}(z) \right\rvert =\left\lvert \mathrm{Res}_{z=\lambda_j}\frac{\prod_{k=1}^{2m} (z-z_k)}{il\prod_{k=1}^{m}       
		(z^2-\lambda_k^2)}\right\rvert 
	&=\left\lvert  \frac{\prod_{k=1}^{2m} (\lambda_j-z_k)}{2l\lambda_j \prod_{\substack{1\leq k \leq m\\
				k\neq j}}  (\lambda_k^2-\lambda_j^2) }\right\rvert.
			\\
			\label{wjsw-new}
			=\left\lvert  \mathrm{Res}_{z=-\lambda_j} \textbf{m}(z) \right\rvert
			&=\left\lvert  \frac{\prod_{k=1}^{2m} (\lambda_j+z_k)}{2l\lambda_j \prod_{\substack{1\leq k \leq m\\
						k\neq j}}  (\lambda_k^2-\lambda_j^2) }\right\rvert.
\end{align}
Equalities~\eqref{wjsw}, \eqref{wjsw-new}, and \eqref{summmm} yield
\begin{align}		\label{imp2}
		\frac{1}{4^m}\prod_{j=1}^m w_j^2=\frac{\prod_{j,k=1}^{2m}|{z_j}-\bar{z}_k|}{(2l)^{2m} 4^m \prod_{j=1}^m \lambda_j^2 \displaystyle\prod_{1\leq j<k\leq m}|\lambda_j^2-\lambda_k^2|^4},
	\end{align}
		
		Substituting \eqref{rpp2}, \eqref{ppp1}, \eqref{imp2} 
		into \eqref{eq:intermediateDensity} we get \eqref{www1}. 
		

		\item We follow similar line of reasoning as in (i). Suppose $z_j$'s belong to $X_{L,M} \subset X_{2n+1}$.
		By~\cite[Lemma 6.5]{KK}
		\begin{align}\label{vande2}
			\left\lvert \det{\frac{\partial{(\Im\kappa_0,\Re\kappa_1,\ldots,\Im\kappa_{2n-2},\Re\kappa_{2n-1},\Im\kappa_{2n})}}{\partial{(y_1,\ldots, y_L ,x_{L+1},y_{L+1},\ldots, x_{L+M},y_{L+M})}}}\right\rvert=2^M \prod_{1\leq j<k\leq 2n+1}|z_j-z_k|,
		\end{align}	
	and then from~\eqref{dist2} we get
	\begin{align}
		\left\lvert \det{\frac{\partial{(\lambda_1,\ldots, \lambda_{n},w_1, \ldots, w_{n},l)}} {\partial{(y_1,\ldots, y_L ,x_{L+1},y_{L+1},\ldots, x_{L+M},y_{L+M})}}   } \right\rvert
		\\ = 2^{M-n} \frac{\prod_{1\leq j<k\leq 2n+1}|z_j-z_k|}{ l^{n}\prod_{j=1}^n \lambda_j^3 \prod_{1\leq j<k\leq n} |\lambda_j^2-\lambda_k^2|^2}.
	\end{align}

		Combining part $(ii)$ of Theorem \ref{thmmm}, and \eqref{vande2} we obtain the joint density of $z_j$'s:
		\begin{align}
			\frac{n!}{2^M M! L!}& \frac{2^{M-n} \prod_{1\leq j<k\leq 2n+1}|z_j-z_k| }{ l^{n}\prod_{j=1}^n \lambda_j^3 \prod_{1\leq j<k\leq n} |\lambda_j^2-\lambda_k^2|^2} \frac{2^n\prod_{j=1}^n\lambda_j}{h_{\beta,n,a}}\prod_{j=1}^n\lambda_j^{\beta a}e^{-\lambda_j^2/2} \prod_{1\leq j<k\leq n}|\lambda_k^2-\lambda_j^2|^\beta   \nonumber\\
			&   \times  \frac{w_0^{\beta(m-n)/2-1}}{\Gamma{(\beta(m-n)/2)}} \times\Gamma(\beta m/2)\prod_{j=1}^n\frac{w_j^{\beta/2-1}}{\Gamma(\beta/2)}
			\times  \prod_{j=1}^L \,dy_j  \prod_{j=L+1}^{L+M} (dx_{j} dy_{j}). 
			\label{longg}
		\end{align}
		Substituting
		\begin{align}\label{imp11}
			\sum_{j=1}^n \lambda_j^2&= \frac12\left(l^2+\sum_{j=1}^{2n+1} {z_j^2}\right),\\
		\label{w0}
			w_0&=\frac{\displaystyle\prod_{j=1}^{2n+1}|z_j|}{l\displaystyle\prod_{j=1}^n|\lambda_j|^2},\\
		\label{imp22}
			\prod_{j=1}^n w_j^2&=\frac{\prod_{j,k=1}^{2n+1}|{z_j}-\bar{z}_k|}{l^{2n} 2^{2n+1}\prod_{j=1}^{2n+1}|z_j| \prod_{j=1}^n \lambda_j^4 \displaystyle\prod_{1\leq j<k\leq n}|\lambda_j^2-\lambda_k^2|^4},
		\end{align}
		into \eqref{longg} we get \eqref{www2}.
		
		
	\end{enumerate}
\end{proof}


	
	
\end{document}